\documentclass{article}

\usepackage{amsmath,amssymb,amscd,amsthm,amsfonts,amstext,amsbsy,mathrsfs,hyperref,upgreek,mathtools,stmaryrd,enumitem,bbm}
\usepackage{tikz}
\usepackage{arrowchasing2}
\usepackage{mathdots,txfonts,mathabx}
\usetikzlibrary{shapes, positioning, decorations.pathreplacing,arrows.meta}
\theoremstyle{theorem}
\newtheorem{theorem}{Theorem}

\theoremstyle{definition}

\allowdisplaybreaks

\makeatletter
\@addtoreset{footnote}{page}
\makeatother

\begin{document}
	
	\title{Arrow-chasing in Pascal's triangle -- Visual proofs for summation formulas involving binomial coefficients}
%	\markright{Arrow-chasing in Pascal's triangle}
	\author{Regula Krapf}  %DO NOT FILL IN AUTHOR'S NAMES UNTIL YOU RECIEVE YOUR PROVISIONAL ACCEPT LETTER. SUBMISSIONS TO THE MONTHLY ARE DOUBLE BLIND.
%	
%	\author{Regula Krapf\\               %%%% Leave ALL of these as is in your initial submission
%		\scriptsize University of Bonn\\    %%%% to allow for double blind reviewing.
%		%affiliation line 2\\                %%%% They should be filled in when you are submitting
%		krapf@math.uni-bonn.de}   
	
	\maketitle
	
	\begin{abstract}
		This article demonstrates, using numerous examples of varying complexity, how one can visually prove summation formulas involving binomial coefficients by exclusively using the recurrence relation for binomial coefficients and its illustration through arrows in Pascal's triangle. The method developed for this purpose, which we call `arrow-chasing', is elementary and it is accessible to a very broad audience.
	\end{abstract}
	
	Binomial coefficients play an important role in combinatorics and in probability theory. Although today the standard definition involves factorials, 
%	$$\binom nk:=\frac{n!}{k!(n-k)!}$$
	historically, binomial coefficients were first considered as entries in the so-called \textit{Pascal triangle}. Pascal's triangle was known in many cultures, e.g., by the Persian mathematician Al-Karaji (953--1029), by the Chinese mathematician Yang Hui (1238--1298) and by the European mathematician Jordanus de Nemore (fl. 13th century). Binomial coefficients can thus be defined as numbers $\binom  nk$ for $0\leq k\leq n$ which satisfy the following recurrence relation, called \textit{Pascal's rule}:
	\begin{align*}
		\binom{n}k&=\binom {n-1}k+\binom {n-1}{k-1}\text{ with }\binom n0=\binom nn=1
	\end{align*}
	In addition, we define $\binom nk$ to be $0$ for $k>n$ and for $k<0$.
	It is a standard exercise to show that our definition is equivalent to the one involving factorials and, moreover, that $\binom nk$ counts the number of ways to choose $k$ out of $n$ objects.
	Using Pascal's rule (see Figure~\ref{fig:Pascal rule}) one obtains Pascal's triangle by placing $1$s in the outermost entries of a number triangle, and then filling each inner entry with the sum of the two numbers diagonally above (see Figure~\ref{fig:Pascal triangle}). By construction, Pascal's triangle is symmetrical.

	\begin{figure}[h!]\centering	\begin{minipage}{0.3\textwidth}
		\vspace*{1.6cm}
			
	\begin{tikzpicture}[scale=1.3]
		\tikzset{hexagon/.style={draw, regular polygon, regular polygon sides=6, minimum size=1.3cm, inner sep=0pt, rotate=30}}
		\drawHexNode00{white}\drawHexNode01{white}\drawHexNode11{white}
		\drawhexarrowsmult00011
		\drawhexarrowsmult01101
		\hextext00{\binom {n-1}{k-1}}\hextext01{\binom {n-1}k}
		\hextext11{\binom nk}
	\end{tikzpicture}
	\caption{\label{fig:Pascal rule}Pascal's rule}

	\end{minipage}\quad
	\begin{minipage}{0.45\textwidth}
		\begin{tikzpicture}[scale=0.8]
			\tikzset{hexagon/.style={draw, regular polygon, regular polygon sides=6, minimum size=0.8cm, inner sep=0pt, rotate=30}}
			\foreach \i in {0,1,2,3,4}{\drawHexNode4{\i}{black!10}}
			\foreach \i in {2,3,5}{\drawHexNode{\i}2{black!30}}
			\drawHexNode42{black!22}
		\drawPascalTriangle5
		\hextext00{\binom00}\hextext10{\binom10}\hextext11{\binom11}\hextext20{\binom20}\hextext21{\binom21}\hextext22{\binom22}
		\hextext30{\binom30}\hextext31{\binom31}\hextext32{\binom32}\hextext33{\binom33}
		\hextext40{\binom40}\hextext41{\binom41}\hextext42{\binom42}\hextext43{\binom43}\hextext44{\binom44}
		\hextext50{\binom50}\hextext51{\binom51}\hextext52{\binom52}\hextext53{\binom53}\hextext54{\binom54}\hextext55{\binom55}
	\end{tikzpicture}\caption{\label{fig:Pascal triangle}Pascal's triangle}
	\end{minipage}
\end{figure}

In the following, we prove various binomial identities using only Pascal's rule and its visualization in Pascal's triangle -- a method we call \textit{arrow-chasing}.

\section{Sums of rows and columns}
We begin with summation formulas for rows and finite parts of columns.  
The \textit{$n$-th row} of Pascal's triangle is the set of entries of the form $\binom{n}{k}$ for $k\in\{0, 1, \ldots, n\}$ (see Figure~\ref{fig:Pascal triangle}, marked light gray).  
Similarly, the \textit{$n$-th column} of Pascal's triangle consists of entries $\binom{n}{k}$ for $n\in\{k, k+1, \ldots\}$ (see Figure~\ref{fig:Pascal triangle}, marked dark gray).

	\begin{theorem}\label{thm:row sum}
	For each $n\in\mathbb N$ we have:
	$$\sum_{k=0}^n\binom nk=2^n$$
\end{theorem}
In Pascal's triangle, this means that the sum of the entries in the $n$-th row is given by $2^n$. Standard proofs of this identity use the binomial theorem, algebraic manipulations or combinatorial arguments. However, one can directly prove it using Pascal's triangle.

	\begin{figure}[h!]\centering
	%	\begin{minipage}{0.49\textwidth}
		\begin{tikzpicture}[scale=1]  % Scale set to 1
			% Define a hexagon shape with a smaller size
			\tikzset{hexagon/.style={draw, regular polygon, regular polygon sides=6, minimum size=1cm, inner sep=0pt, rotate=30}}
			\drawPascalTrap{5}{6}
			\foreach \i in {0,1,...,6}{\drawHexNode6{\i}{black!10}}
			\foreach \i in {0,1,2,3,4,5}{\drawhexarrowsmult5{\i}111}
		\end{tikzpicture}\caption{Row sum in Pascal's triangle\label{fig:row sum}}\end{figure}
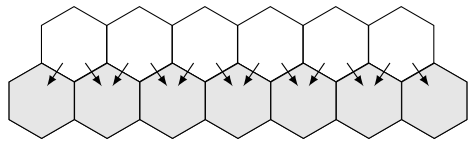
	
		\begin{proof}[Proof of Theorem \ref{thm:row sum}]
		We interpret the left-hand side as the sum of all entries in the $n$-th row of Pascal's triangle. In order to obtain that row sum, we add each entry in the row above twice (see Figure \ref{fig:row sum}). Hence, the row sum is twice the previous row sum. Since the row sum in row $0$ is $1=2^0$, this proves the claim.
	\end{proof}
	
		Note that our proof of Theorem~\ref{thm:row sum} implicitly uses induction. However, compared to the standard induction proof using algebraic manipulations, this visual argument is much easier to follow and it provides more insight into the question \textit{why} this identity holds.	Moreover, it is accessible to a much broader audience\footnote{In fact, the author discussed some of the arguments presented in this paper with primary school children who even managed to come up with their own arrow-chasing proofs.}.
	Almost the same proof can be used to show that if we only add every second entry in the $n$-th row of Pascal's triangle we obtain exactly the sum of all entries in the row above, which is equal to $2^{n-1}$ by Theorem~\ref{thm:row sum} (see Figure \ref{fig:half row sum}). This holds whether we sum only the even-indexed entries (left) or only the odd-indexed ones (right), yielding:
	
	$$\sum\limits_{k\geq0}\binom n{2k}=\sum\limits_{k\geq0}\binom n{2k+1}=2^{n-1}$$
	
	%\end{minipage}\;\;
	\begin{figure}[h!]\centering
		%\begin{minipage}{0.49\textwidth}
		\begin{tikzpicture}[scale=1]  % Scale set to 1
			% Define a hexagon shape with a smaller size
			\tikzset{hexagon/.style={draw, regular polygon, regular polygon sides=6, minimum size=1cm, inner sep=0pt, rotate=30}}
			\drawPascalTrap56
			\foreach \i in {0,2,4,6}{
				\drawHexNode6{\i}{black!10}
			}
			\foreach \i in {0,2,4}{\drawhexarrowsmult5{\i}101\drawhexarrowsmult5{\i+1}011}
		\end{tikzpicture}\;\;
		\begin{tikzpicture}\tikzset{hexagon/.style={draw, regular polygon, regular polygon sides=6, minimum size=1cm, inner sep=0pt, rotate=30}}
			\drawPascalTrap56
			\foreach \i in {1,3,5}{\drawHexNode6{\i}{black!10}\drawhexarrowsmult5{\i-1}011\drawhexarrowsmult5{\i}101}
		\end{tikzpicture}
		\caption{Sum of half of the entries in a row\label{fig:half row sum}}
		%\end{minipage}
	\end{figure}
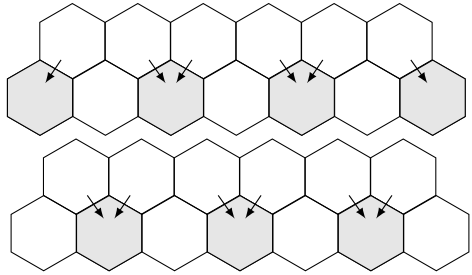

	While these particular proofs are not new (they have previously been published by \'Angel Plaza in \cite{MR3654845} and \cite{MR4153216}), the method of drawing arrows in Pascal's triangle -- which we denote by \textit{arrow-chasing} -- offers a powerful tool for discovering a wide range of new proofs of identities involving binomial coefficients.
	
	Next, we consider sums of entries in the same column of Pascal's triangle. This yields the so-called \textit{hockey stick identity}:
	
		\begin{theorem}[Hockey stick identity]\label{thm:hockey stick}
		For all $n,m\in\mathbb N$ such that $m\leq n$ we have:
		$$\sum\limits_{k=m}^n\binom km=\binom{n+1}{m+1}$$	
	\end{theorem}
	
	\begin{proof}
		As displayed in Figure~\ref{fig:hockey stick}, we can interpret $\sum_{k=m}^n\binom km$ as the sum of the entries in the $m$-th column in Pascal's triangle from the $m$-th to the $n$-th row (marked light gray). If we add these entries by successively using Pascal's rule and the fact that $\binom mm=1=\binom{m+1}{m+1}$, we obtain $\binom{n+1}{m+1}$ (marked dark gray).
	\end{proof}
	
	\begin{figure}[h!]\centering
	\begin{tikzpicture}
		\drawPascalTriangle6\drawHexNode63{black!30}
		\foreach \i in {2,3,4,5}{\drawHexNode{\i}2{black!10}\drawhexarrowsmult{\i}2011}
		\drawhexarrowsmult33101\drawhexarrowsmult43101\drawhexarrowsmult53101
		\hextext63{\binom{n+1}{m+1}}\hextext22{\binom mm}\hextext32{\binom{m+1}m}\hextext42{\raisebox{0.2cm}{\text{$\iddots$}}}\hextext52{\binom nm}
	\end{tikzpicture}\caption{Proof of the hockey stick identity \label{fig:hockey stick}}
\end{figure}
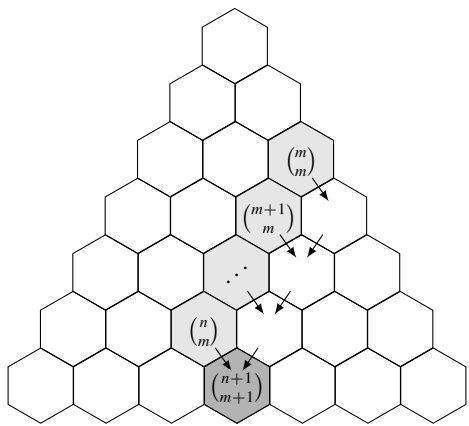

%	When searching for the hockey stick identity in the literature, one usually finds the visualization (without arrows!) in the shape of a hockey stick in Pascal's triangle. However, the identity is still commonly proved using induction, alongside algebraic manipulations, or via a combinatorial argument. In comparison, our pictorial argument, which -- unlike in the proof of Theorem~\ref{thm:row sum} -- does not use induction, is both more direct and intuitive.  
	
In standard literature, the hockey stick identity is often illustrated (without arrows) as a hockey stick shape in Pascal's triangle. Yet, proofs typically rely on induction, algebraic manipulation, or combinatorial arguments. 
In contrast to Theorem~\ref{thm:row sum}, this diagram captures the whole identity at once: the arrows and dots show that every term is accounted for in a single visual sweep. Moreover, this approach offers the advantage of conveying the result in an intuitive visual way, making it accessible without formal inductive reasoning.

	%In contrast, our pictorial proof -- unlike that of Theorem~\ref{thm:row sum} -- avoids induction and is both more direct and intuitive.

%	\begin{figure}[h!]\centering
%		\begin{tikzpicture}
%			\drawPascalTriangle6\drawHexNode63{black!35}
%			\foreach \i in {2,3,4,5}{\drawHexNode{\i}{\i-2}{black!10}\drawhexarrowsmult{\i}{\i-2}101}
%			\drawhexarrowsmult30011\drawhexarrowsmult41011\drawhexarrowsmult52011
%			\hextext63{\text{\tiny{$\binom{m+n+1}{n}$}}}\hextext20{\binom m0}\hextext31{\binom{m+1}1}\hextext42{\raisebox{0.2cm}{\text{$\iddots$}}}\hextext53{\binom {m+n}n}
%		\end{tikzpicture}\caption{Proof of the hockey stick identity \label{fig:hockey stick}}
%	\end{figure}

\section{Weighted sums of rows}	
In the following, we will consider weighted sums of binomial coefficients in a row. Simple examples of such weighted sums include
$$\sum\limits_{k=0}^nk\cdot\binom nk\quad\text{or}\quad\sum\limits_{k=0}^n(-1)^k\binom nk.$$
At this point, we move beyond familiar territory.
First, note that in Figure~\ref{fig:row sum}, it was not necessary to write the values of the binomial coefficients into the corresponding cells. This works because each value is uniquely determined by its position in Pascal’s triangle. This observation allows us to place the weight of each binomial coefficient into the now-empty cells instead.

%Second, in contrast to the downward-moving Pascal rule, when dealing with weights in a weighted sum, the process is reversed -- we move \textit{upward} in the triangle.
%To formalize this, we introduce what we call the \textit{weight rule}: Given a weighted row sum of the form $\sum_{k=0}^na_k\binom  nk$,
%we can compute the corresponding weights in the row above (i.e., for $n-1$) such that the total weighted sum remains unchanged. According to the weight rule, each weight in the upper row is obtained by summing the two adjacent weights directly below it (see Figure~\ref{fig:weight rule}). This mirrors Pascal’s rule, but in reverse. If we use the weight rule, we change the orientation of the arrows.

Second, unlike the downward-moving Pascal rule, when dealing with weights in a weighted sum, the process is reversed -- we move {upward} in the triangle.
To formalize this, we introduce the \textit{weight rule}: Given a weighted row sum $\sum_{k=0}^n a_k \binom{n}{k}$, we can compute weights for the row above ($n-1$) so that the total weighted sum stays unchanged.  
According to the weight rule, each weight in the upper row equals the sum of the two adjacent weights directly below it (see Figure~\ref{fig:weight rule}). This mirrors Pascal' rule, but in reverse. If we use the weight rule, we change the orientation of the arrows in order to emphasize that we work upward.
\begin{figure}[h!]\centering
	\begin{tikzpicture}[scale=1.4]
		\tikzset{hexagon/.style={draw, regular polygon, regular polygon sides=6, minimum size=1.4cm, inner sep=0pt, rotate=30}}
		\drawPascalTrap45
		\foreach \i/\j in {0/1,1/2,2/3,3/4,4/5}{
			\drawhexarrowsmultrev4{\i}111
			\hextext4{\i}{a_\i+a_\j}
			\hextext5{\i}{a_\i}
		}
		\hextext55{a_5}
	\end{tikzpicture}
	\caption{Weight rule\label{fig:weight rule}}
\end{figure}
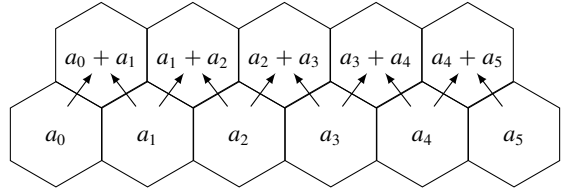
Note that -- in contrast to Pascal’s rule -- the weight rule must always be applied to the entire row. If a particular entry is not meant to contribute to the sum, we can simply assign it a weight of $0$ and leave the cell empty.

Now, we are ready to present some examples of weighted row sums. We start with the following well-known theorem:

	\begin{theorem}\label{thm:weighted binomial sum}
	For  every $n\in\mathbb N$ we have:
	$$\sum\limits_{k=0}^nk\cdot\binom nk=n\cdot2^{n-1}$$
\end{theorem}

\begin{proof}
We interpret the left-hand side as a weighted row sum in the $n$-th row of Pascal's triangle (see Figure~\ref{fig:weighted row sum}, left), where the binomial coefficient $\binom nk$ is equipped with weight $k$. Using the weight rule we add adjacent weights in order to obtain a weighted sum in the $(n-1)$-st row with weights $k+(k+1)$. Now, we use the symmetry of Pascal's triangle and we rearrange the weights of equal binomial coefficients as depicted in Figure~\ref{fig:weighted row sum}, right): The binomial coefficients $\binom {n-1}k$ and $\binom 
{n-1}{n-1-k}$ have weights $k+(k+1)$ and $(n-1-k)+(n-k)$. By symmetry we can swap the weights in order to obtain $k+(n-k)=n$ and $(k+1)+(n-1-k)=n$. Then every entry has weight $n$, so we get $n$ times the row sum of the $(n-1)$-st row, which is $n\cdot2^{n-1}$ by Theorem~\ref{thm:row sum}.
\end{proof}

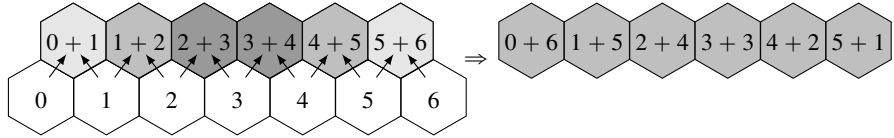
\begin{figure}[h!]\centering
	\begin{tikzpicture}
		\drawPascalTrap56
		%\foreach \i in {0,1,2,3,4,5,6}{\drawHexNode6{\i}{black!10}}
	\drawHexNode50{black!10}\drawHexNode55{black!10}\drawHexNode51{black!25}\drawHexNode54{black!25}\drawHexNode53{black!40}\drawHexNode52{black!40}
			\foreach \i/\j in {0/1,1/2,2/3,3/4,4/5,5/5}{
			\drawhexarrowsmultrev5{\i}111}
			\foreach \i in {0,1,2,3,4,5,6}{\hextext6{\i}{\i}}
			\hextext50{0+1}\hextext51{1+2}\hextext52{2+3}\hextext53{3+4}\hextext54{4+5}\hextext55{5+6}
			\hextext{5.3}{6.3}{\Rightarrow}
			\foreach \i in {7,8,9,10,11,12}{\drawHexNode5{\i}{black!25}}
			\hextext57{0+6}\hextext58{1+5}\hextext59{2+4}\hextext5{10}{3+3}\hextext5{11}{4+2}\hextext5{12}{5+1}
	\end{tikzpicture}\caption{Proof of Theorem~\ref{thm:weighted binomial sum}\label{fig:weighted row sum} for $n=6$}
\end{figure}

We can also find formulas for weighted sums where we only add every second entry of each row:

	\begin{theorem}\label{thm:weighted row sum 2}
	For each $k\in\mathbb N$ we have:
	\begin{align*}
		\sum\limits_{k\geq0}k\cdot\binom n{2k}=n\cdot2^{n-3}\quad\text{and}\quad \sum\limits_{k\geq0}k\cdot\binom n{2k+1}=(n-2)\cdot2^{n-3}
	\end{align*}
\end{theorem}

\begin{proof}
	We consider the $n$-th row of Pascal's triangle, where we add only the weighted even-indexed or the weighted odd-indexed entries (see Figure \ref{fig:weighted row sum 2}). 
	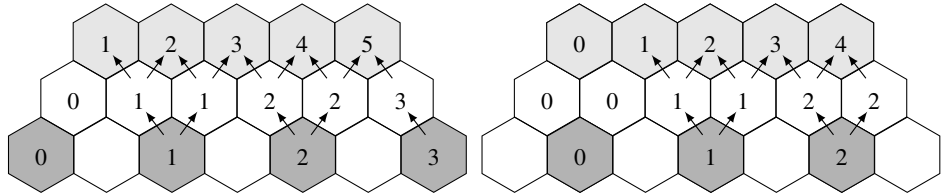
\begin{figure}[h!]
		\centering
		\begin{tikzpicture}
			\tikzset{hexagon/.style={draw, regular polygon, regular polygon sides=6, minimum size=1cm, inner sep=0pt, rotate=30}}
			\drawHexNode50{black!10}
			\foreach \i in {0,2,4,6}{\drawHexNode7{\i}{black!30}}
			\foreach \i in {1,2,3,4}{\drawHexNode5{\i}{black!10}\drawhexarrowsmultrev5{\i}111,\drawHexNode6{\i}{none}}
			\foreach \i/\j in {6/0,6/5,7/1,7/3,7/5}{\drawHexNode{\i}{\j}{none}}
			\foreach \i in {1,3,5}{\drawhexarrowsmultrev6{\i}011}
			\drawhexarrowsmultrev50011\drawhexarrowsmultrev62101\drawhexarrowsmultrev64101
			\hextext70{0}\hextext72{1}\hextext74{2}\hextext76{3}
			\hextext60{0}\hextext61{1}\hextext62{1}\hextext63{2}\hextext64{2}\hextext65{3}\hextext50{1}\hextext51{2}\hextext52{3}\hextext53{4}\hextext54{5}
		\end{tikzpicture}\;\;\begin{tikzpicture}
			\tikzset{hexagon/.style={draw, regular polygon, regular polygon sides=6, minimum size=1cm, inner sep=0pt, rotate=30}}
			\drawPascalTrap{6}{4}
			\drawHexNode40{black!10}
			\foreach \i in {1,2,3,4} {\drawHexNode4{\i}{black!10}}
			\foreach \i in {1,3,5} {\drawHexNode6{\i}{black!30}}
			\foreach \i/\j in {4/2,4/3,4/4}{\drawhexarrowsmultrev{\i}{\j}111}\drawhexarrowsmultrev41011
			\drawhexarrowsmultrev52011\drawhexarrowsmultrev53101\drawhexarrowsmultrev54011\drawhexarrowsmultrev55101
			\hextext61{0}\hextext63{1}\hextext65{2}
			\hextext50{0}\hextext51{0}\hextext52{1}\hextext53{1}\hextext54{2}\hextext55{2}
			\hextext40{0}\hextext41{1}\hextext42{2}\hextext43{3}\hextext44{4}
		\end{tikzpicture}\caption{Proof of Theorem \ref{thm:weighted row sum 2}\label{fig:weighted row sum 2}}
	\end{figure}
	Using the weight rule, we move to row $n-2$.
	The second identity then clearly follows from Theorem \ref{thm:weighted binomial sum} for $n-2$ since we obtain weight $k$ in the $k$-th entry of the $(n-2)$-nd row. The weighted sum of the even-indexed entries of the $n$-th row is similar, but instead we obtain weight $k+1$ in the $k$-th entry of the $(n-2)$-nd row. Hence, we obtain the same sum as in the second identity but we additionally have to add the row sum in the $(n-2)$-nd row, which equals $2^{n-2}$ by Theorem~\ref{thm:row sum}. Therefore, in total we obtain $(n-2)\cdot2^{n-3}+2^{n-2}=n\cdot2^{n-3}$, which corresponds to the right-hand side.
\end{proof}

Next, we choose binomial coefficients themselves as weights:

	\begin{theorem}[Lagrange's identity]\label{thm:Lagrange}
	For every $n\in\mathbb N$ we have:
	$$\sum\limits_{k=1}^n\binom nk^2=\binom{2n}n$$
\end{theorem}

\begin{proof}The proof is depicted in Figure~\ref{fig:Lagrange 2} for $n=4$.
%We start with the right-hand side of our identity and view this as then $n$-th entry in the $(2n)$-th row of Pascal's triangle with weight $1$. Using the weight rule, we move upwards in Pascal's triangle until we reach the $n$-th row. Clearly, since we add adjacent weights in order to obtain the weight placed diagonally above, the weights generate an inverted version of Pascal's triangle. This means that in the $n$-th row each weight coincides with the value of its entry, so that we get the left-hand side. Moreover, by the weight rule, the row sum remains invariant and is hence always equal to $\binom{2n}n$, hence so is the weighted sum in the $n$-th row.
We begin with the right-hand side of the identity and interpret it as the central entry in the $(2n)$-th row of Pascal's triangle, which we assign a weight of $1$. Using the weight rule, we move upwards through Pascal's triangle toward the $n$-th row. Since each weight is obtained by summing the two weights directly below it, this process creates an inverted version of Pascal's triangle.
As we reach the $n$-th row, each weight aligns exactly with the value of the corresponding entry in that row. This gives us the left-hand side of the identity. Moreover, by the weight rule, the weighted sum remains constant throughout the rows and is therefore always equal to $\binom{2n}{n}$. Thus, the weighted sum in the $ n$-th row given by $\sum_{k=0}^n\binom nk^2$ is also equal to $\binom{2n}{n}$, completing the argument.

\end{proof}

	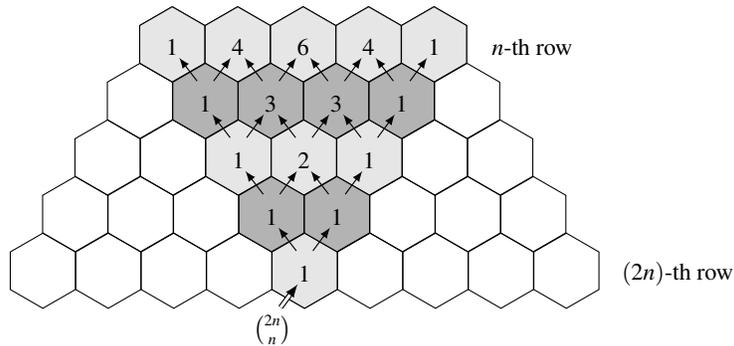
\begin{figure}[h!]\centering
	\begin{tikzpicture}
		\tikzset{hexagon/.style={draw, regular polygon, regular polygon sides=6, minimum size=1cm, inner sep=0pt, rotate=30}}
		
		% Create Pascal's Triangle in hexagon shape
		\drawPascalTrap48
		\foreach \j in {0,1,...,4}{
			\drawHexNode{4}{\j}{black!10}
		}
		\foreach \i/\j in {6/2,6/3,6/4,8/4}{
			\drawHexNode{\i}{\j}{black!10}
		}
		\foreach \i/\j in {5/1,5/2,5/3,5/4,7/3,7/4}{
			\drawHexNode{\i}{\j}{black!30}}
		\foreach \i in {4,5,6,7}{
			\drawhexarrowsmultrev{\i}{\i-4}{0}{1}{1}
		}
		\foreach \i in {4,5,6,7}{
			\drawhexarrowsmultrev{\i}{4}{1}{0}{1}
		}
		\drawhexarrowsmultrev{6}{3}{1}{1}{1}
		\drawhexarrowsmultrev{5}{2}{1}{0}{1}
		\drawhexarrowsmultrev{5}{2}{0}{1}{1}
		\drawhexarrowsmultrev{5}{3}{1}{0}{1}
		\drawhexarrowsmultrev{5}{3}{0}{1}{1}
		\drawhexarrowsmultrev42111
		\drawhexarrowsmultrev41101
		\drawhexarrowsmultrev41011
		\drawhexarrowsmultrev43101
		\drawhexarrowsmultrev43011
		\hextext{8}{4}1	\hextext731
		\hextext741
		\hextext621
		\hextext632
		\hextext641
		\hextext511
		\hextext523
		\hextext533
		\hextext541
		\hextext401\hextext414\hextext426\hextext434\hextext441
		\hextext4{5.5}{\text{$n$-th row}}
		\hextext8{9.7}{\text{$(2n)$-th row}}
		\hextext94{\binom{2n}n}
		\drawhexarrowsmultrev84102
	\end{tikzpicture}\caption{Proof of Lagrange's identity\label{fig:Lagrange 2} for $n=4$}
\end{figure}

Essentially the same proof can be used to derive the more general formula called \textit{Chu-Vandermonde identity}. This identity is often named only after the French mathematician Alexandre-Th\'eophile Vandermonde (who published its proof in 1772), but it was already known in 1303 by the Chinese mathematician Zhu Shijie (see \cite{MR2768578}).

\begin{theorem}[Chu-Vandermonde identity]
	For $i,m,n\in\mathbb N$ such that $l\leq m,n$ we have:
	$$\sum_{k=0}^l\binom nk\cdot\binom m{l-k}=\binom{m+n}l$$
\end{theorem}
The weight rule also applies to negative weights, allowing us to prove identities that involve alternating sums of binomial coefficients.

	\begin{theorem}\label{thm:alternating sum}
	For every $n\in\mathbb N$ with $n>0$ we have:
	$$\sum\limits_{k=0}^n(-1)^k\binom nk=0$$
\end{theorem}

\begin{proof}
	According to the weight rule, each entry in the $(n-1)$-th row is both added and subtracted once (see Figure~\ref{fig:alternating sum}). This means that each entry in that row has weight $0$ and hence the weighted row sum is also equal to $0$. Since by construction the weighted sum in the $n$-th row is equal to that of the $(n-1)$-st row, the claim follows.
	\begin{figure}[h!]\centering
		\begin{tikzpicture}[scale=1.3]
			\tikzset{hexagon/.style={draw, regular polygon, regular polygon sides=6, minimum size=1.3cm, inner sep=0pt, rotate=30}}
			\foreach \i in {0,1,...,4}{
				\drawHexNode4{\i}{black!30}
			}
			\foreach \i in {0,1,...,5}{
				\drawHexNode5{\i}{black!10}
			}
			\drawhexarrowsmultrev40101
			\drawhexarrowsmultrev40011
			\drawhexarrowsmultrev41101
			\drawhexarrowsmultrev41011
			\drawhexarrowsmultrev42101
			\drawhexarrowsmultrev42011
			\drawhexarrowsmultrev43101
			\drawhexarrowsmultrev43011
			\drawhexarrowsmultrev44101
			\drawhexarrowsmultrev44011
			\hextext50{1}		
			\hextext51{-1}
			\hextext52{1}
			\hextext53{-1}
			\hextext54{1}
			\hextext55{-1}
			\hextext40{1-1}
			\hextext41{-1+1}
			\hextext42{1+-1}
			\hextext43{-1+1}
			\hextext44{1-1}
			%		\foreach \i in {0,1,...,4}{
				%		\hextext4{\i}{$0$}
				%		}
			%\drawhexarrowsmult45105
		\end{tikzpicture}\caption{Proof of Theorem \ref{thm:alternating sum}\label{fig:alternating sum}}
	\end{figure}
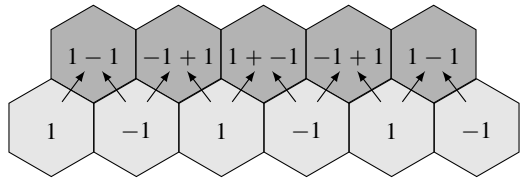
\end{proof}

	Essentially the same proof works also for the following formula which combines the weights considered in Theorems~\ref{thm:weighted binomial sum} and~\ref{thm:alternating sum}:
$$\sum_{k=0}^n(-1)^kk\binom nk=0.$$
If we add adjacent entries according to the weight rule, we obtain the weighted sum $\sum_{k=0}^{n-1}(-1)^{k+1}\binom{n-1}k$ in the row above, which is equal to $0$ by Theorem~\ref{thm:alternating sum}.

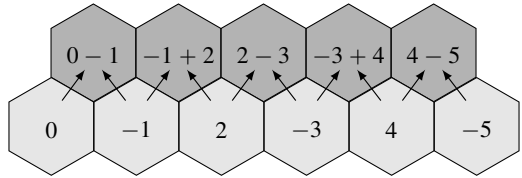
\begin{figure}[h!]\centering
	\begin{tikzpicture}[scale=1.3]
		\tikzset{hexagon/.style={draw, regular polygon, regular polygon sides=6, minimum size=1.3cm, inner sep=0pt, rotate=30}}
		\foreach \i in {0,1,...,4}{
			\drawHexNode4{\i}{black!30}
		}
		\foreach \i in {0,1,...,5}{
			\drawHexNode5{\i}{black!10}
		}
\foreach \i in {0,1,2,3,4}{\drawhexarrowsmultrev4{\i}111}

		\hextext50{0}		
		\hextext51{-1}
		\hextext52{2}
		\hextext53{-3}
		\hextext54{4}
		\hextext55{-5}
		\hextext40{0-1}
		\hextext41{-1+2}
		\hextext42{2-3}
		\hextext43{-3+4}
		\hextext44{4-5}
		%		\foreach \i in {0,1,...,4}{
			%		\hextext4{\i}{$0$}
			%		}
		%\drawhexarrowsmult45105
	\end{tikzpicture}\caption{Alternating weighted row sum\label{fig:alternating weighted row sum}}
\end{figure}

Note that the weights in Figures~\ref{fig:alternating weighted row sum} and \ref{fig:Alternating binomial row sum} correspond to the alternating values of the $0$-th and first column of Pascal's triangle. This suggests that we can generalize this further by replacing the weights $(-1)^k=(-1)^k\binom k0$ and $(-1)^k\cdot k=(-1)^k\binom k1$ by $(-1)^k\binom km$ for some $m\in\mathbb N$. In fact, we get the following general statement:

	\begin{theorem}\label{thm:Alternating binomial row sum}
	For natural numbers $m<n$ we have:
	$$\sum_{k=m}^n(-1)^k\binom nk\binom km=0$$
	%	where $\delta_{n,m}$ is the Kronecker delta defined as $1$ if $n=m$ and $0$ otherwise.
\end{theorem}

%In our proof, rather than writing multi-arrows, we just write single arrows and write the multiplicity of each binomial coefficient in its field.

Figure~\ref{fig:Alternating binomial row sum} shows the case $n=8$ and $m=3$.

\begin{proof}
	We consider the $n$-th row of Pascal's triangle, starting at the entry $\binom nm$ (see Figure \ref{fig:Alternating binomial row sum}). Now, the weight of each binomial coefficient $\binom nk$ is given by $(-1)^k\binom km$. Thus, the weights correspond to the alternating entries of the $m$-th column. If we add two consecutive weights, we obtain by Pascal's rule:
	$$(-1)^k\binom km+(-1)^{k+1}\binom{k+1}m=(-1)^{k+1}\binom k{m-1}$$
	Hence, the sums of the weights in the row above correspond to the alternating entries of the $(m-1)$-th column, and so on. In the $(n-m)$-th row, the entries are therefore given by the alternating entries of the $0$-th column, which are $\pm1$. Thus, the left-hand side is equal to the alternating sum of the entries of the $(n-m)$-th row, which is $0$. 
\end{proof}
Interestingly, the weights in Figure~\ref{fig:Alternating binomial row sum} form part of a rotated Pascal triangle.

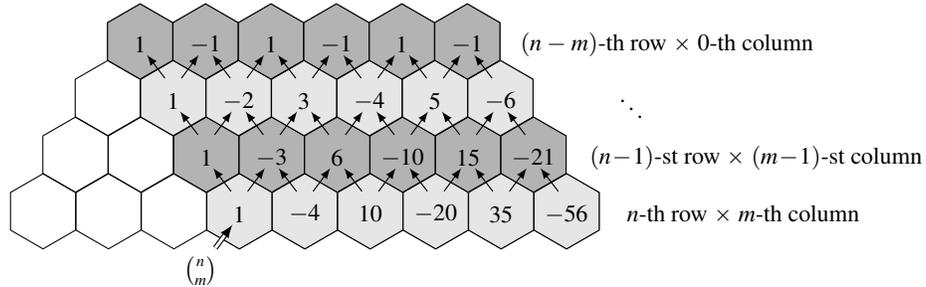
\begin{figure}[h!]\centering
	\begin{tikzpicture}
		\tikzset{hexagon/.style={draw, regular polygon, regular polygon sides=6, minimum size=1cm, inner sep=0pt, rotate=30}}
		\drawPascalTrap58
%		\foreach \i in {0,1,2,3,4}{
%			\drawHexNode{4}{\i}{black!35}
		%	\hextext4{\i}{0}}
		\foreach \i in {2,3,4,5,6,7}{
			\drawHexNode7{\i}{black!30}}
		\foreach \i in {1,2,3,4,5,6}{
			\drawHexNode{6}{\i}{black!10}}
		\foreach \i in {0,1,2,3,4,5}{\drawHexNode{5}{\i}{black!30}}
		\foreach \i in {3,4,5,6,7,8}{\drawHexNode8{\i}{black!10}}
\foreach \i/\j in {5/1,5/2,5/3,5/4,5/5,6/2,6/3,6/4,6/5,6/6,7/3,7/4,7/5,7/6,7/7}{\drawhexarrowsmultrev{\i}{\j}111}\drawhexarrowsmultrev50011\drawhexarrowsmultrev61011\drawhexarrowsmultrev72011
		\hextext72{1}\hextext73{-3}\hextext74{6}\hextext75{-10}\hextext76{15}\hextext77{-21}
		\hextext61{1}\hextext62{-2}\hextext63{3}\hextext64{-4}\hextext65{5}\hextext66{-6}
		\hextext50{1}\hextext51{-1}\hextext52{1}\hextext53{-1}\hextext54{1}\hextext55{-1}
		\hextext831\hextext84{-4}\hextext85{10}\hextext86{-20}\hextext87{35}\hextext88{-56}
		\hextext5{8.05}{\text{$(n-m)$-th row $\times$ $0$-th column}}
	%	\hextext6{8.8}{\text{sixth row $\times$ first column}}
		\hextext7{10.4}{\text{$(n\!-\!1)$-st row $\times$ $(m\!-\!1)$-st column}}
		\hextext8{10.7}{\text{$n$-th row $\times$ $m$-th column}}
		\hextext68{\ddots}
		\hextext9{2.9}{\binom nm}
		\drawhexarrowsmultrev83102
	\end{tikzpicture}\caption{Proof of Theorem \ref{thm:Alternating binomial row sum} for $n=8$ and $m=3$\label{fig:Alternating binomial row sum}}
\end{figure}

The weights are not restricted to integers. We can also prove identities involving weighted sums of binomial coefficients with rational, real or even complex weights. The most prominent example of such an identity is the binomial theorem.

	\begin{theorem}[Binomial Theorem]\label{thm:binomial theorem}
	For all complex numbers $a$ and $b$ we have:
	$$(a+b)^n=\sum\limits_{k=0}^n\binom nka^{n-k}b^k$$
\end{theorem}

\begin{figure}[h!]
	\centering
	\begin{tikzpicture}[scale=1.8]
		\tikzset{hexagon/.style={draw, regular polygon, regular polygon sides=6, minimum size=1.8cm, inner sep=0pt, rotate=30}}
		\foreach \i in {0,1,...,4}{
			\drawHexNode4{\i}{black!10}
		}
		\foreach \i in {0,1,...,5}{
			\drawHexNode5{\i}{black!30}
		}
		\drawhexarrowsmultrev40101
		\drawhexarrowsmultrev40011
		\drawhexarrowsmultrev41101
		\drawhexarrowsmultrev41011
		\drawhexarrowsmultrev42101
		\drawhexarrowsmultrev42011
		\drawhexarrowsmultrev43101
		\drawhexarrowsmultrev43011
		\drawhexarrowsmultrev44101
		\drawhexarrowsmultrev44011
		\hextext40{\text{\tiny{$a^n(a\!+\!b)$}}}		
		\hextext41{\text{\tiny{$a^{n-1}b(a\!+\!b)$}}}
		\hextext42{\text{\tiny{$a^{n-2}b^2(a\!+\!b)$}}}
		\hextext43{\ldots}
		\hextext44{\text{\tiny{$b^n(a\!+\!b)$}}}
		\hextext50{a^{n+1}}
		\hextext51{a^nb}
		\hextext52{a^{n-1}b^2}
		\hextext53{\ldots}
		\hextext54{ab^n}\hextext55{b^{n+1}}
		\hextext5{6.25}{\text{$(n+1)$-st row}}
		\hextext4{5}{\text{$n$-th row}}
	\end{tikzpicture}
	\caption{Proof of the binomial theorem\label{fig:Binomial theorem}}
\end{figure}
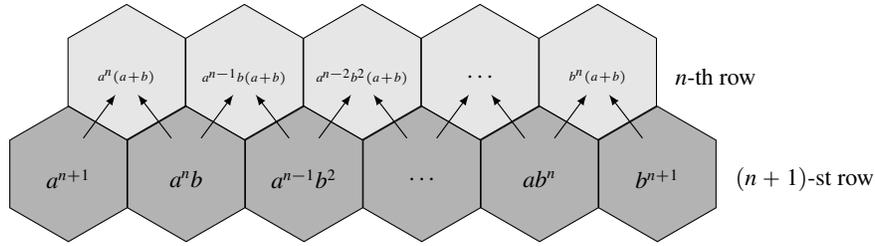

\begin{proof}
Let us consider the right-hand side in the $(n+1)$-st row of Pascal's triangle, equipped with weights $a^{n+1}, a^n b, \ldots, b^{n+1}$. Using the weight rule, we add adjacent weights and get $a^{n+1-k}b^k + a^{n-k}b^{k+1} = a^{n-k}b^k(a + b)$ in the $n$-th row. If we factor out $(a + b)$, the remaining weight sequence is $a^n, a^{n-1}b, \ldots, b^n$. Hence, the weighted row sum in the $(n+1)$-st row is $(a + b)$ times the previous row sum. Since the weighted row sum in row $0$ is $\binom{0}{0}a^0b^0 = 1 = (a + b)^0$, this proves the claim by induction.
%Let us now consider the $(n+1)$-th row in Pascal's triangle, equipped with weights $a^{n+1},a^nb,a^{n-1}b^2,\ldots, b^{n+1}$. Using the weight rule, we add adjacent weights and get $a^{n+1-k}b^k+a^{n-k}b^{k+1}=a^{n-k}b^k(a+b)$. 
%Hence, we can factor out $(a+b)$ in the weight sequence of the $n$-th row, so that we obtain
%$$(a+b)\sum_{k=0}^n\binom nka^{n-k}b^k=\sum_{k=0}^{n+1}\binom{n+1}ka^{n+1-k}b^k.$$
%By induction, the left-hand side is equal to $(a+b)(a+b)^n=(a+b)^{n+1}$.
\end{proof}

Note that this proof is simply a generalization of the proof of Theorem~\ref{thm:row sum}.

%In particular, we again make use of induction.

%Clearly, there are many more options for placing weights in rows of Pascal's triangle. For example, if we consider Fibonacci weights, we can prove the following nice identity:
%$$\sum\limits_{k=0}^n\binom nkF_k=F_{2n}$$
%Our arrow-chasing approach allows for generalizing this classical result in many ways and even prove new binomial (Fibonacci) identities. This shall, however, be the subject of further paper \cite{xx}.

\section{Weighted sums of columns}
In this section, we will explore weighted sums of finite portions of columns in Pascal's triangle. If we choose our weights all to be $1$, we simply obtain the hockey stick rule. Now, let us consider other weight sequences. The easiest option is to choose weight $k$. This gives us the following variant of the hockey stick identity, which was shown in \cite{MR1390415}. However, our proof is a lot shorter than the one presented in \cite{MR1390415}.

\begin{theorem}\label{thm:Hockey variant}
	For all $n,m\in\mathbb N$ such that $m< n$ we have:
	$$\sum\limits_{k=m}^nk\binom km=n\binom{n+1}{m+1}-\binom{n+1}{m+2}$$
\end{theorem}

\begin{proof}
	As illustrated in Figure~\ref{fig:Hockey variant} (left), we begin with the entries $\binom{n+1}{m+1}$ (with weight $n$) and $\binom{n+1}{m+2}$ (with weight $-1$). Using the weight rule, we can determine the weights of the entries in the $m$-th, $(m+1)$-st and $(m+2)$-nd column, as displayed below for $m=1$. Clearly, the weight of $\binom km$ is given by $k$, which establishes the claim. 
\end{proof}

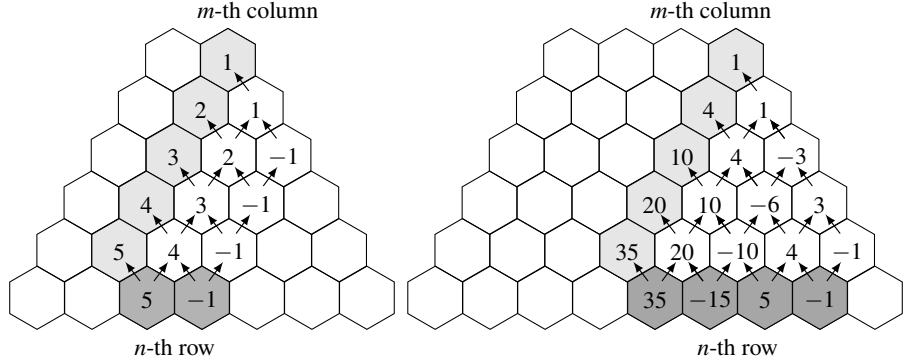
\begin{figure}[h!]
	\centering
	\begin{tikzpicture}[scale=0.84]
		\tikzset{hexagon/.style={draw, regular polygon, regular polygon sides=6, minimum size=0.84cm, inner sep=0pt, rotate=30}}
		\drawPascalTrap16
		\drawHexNode62{black!30}\drawHexNode63{black!30}
		\foreach \i in {1,2,3,4,5}{
			\drawHexNode{\i}{1}{black!10}\drawhexarrowsmultrev{\i}1011}
		\foreach \i in {2,3,4,5}{\drawhexarrowsmultrev{\i}2111}
		\foreach \i in {3,4,5}{\drawhexarrowsmultrev{\i}{3}101}
		\foreach \i in {1,2,3,4,5}{\hextext{\i}1{\i}}\hextext221\hextext322\hextext423\hextext524\hextext625
		\foreach \i in {3,4,5,6}{\hextext{\i}3{-1}}
		\hextext{-0.1}1{\text{$m$-th column}}
		\hextext73{\text{$n$-th row}}
	\end{tikzpicture}\:
			\begin{tikzpicture}[scale=0.84]
		\tikzset{hexagon/.style={draw, regular polygon, regular polygon sides=6, minimum size=0.84cm, inner sep=0pt, rotate=30}}
		\drawPascalTrap38
		\foreach \i in {4,5,6,7} {\drawHexNode8{\i}{black!35}}
		\foreach \i in {3,4,5,6,7}{\drawHexNode{\i}3{black!10}}
		\foreach \i in {4,5,6,7}{\drawhexarrowsmultrev{\i}4011}
		\foreach \i/\j in {4/4,5/4,5/5,6/4,6/5,6/6,7/4,7/5,7/6}{\drawhexarrowsmultrev{\i}{\j}111}
		\drawhexarrowsmultrev77101
		\foreach \i in {3,4,5,6,7}{\drawhexarrowsmultrev{\i}3011}
		\hextext331\hextext434\hextext53{10}\hextext63{20}\hextext73{35}
		\hextext441\hextext544\hextext64{10}\hextext74{20}\hextext84{35}
		\hextext55{-3}\hextext65{-6}\hextext75{-10}\hextext85{-15}
		\hextext663\hextext764\hextext865\hextext77{-1}\hextext87{-1}
		\hextext{1.9}2{\text{$m$-th column}}
		\hextext96{\text{$n$-th row}}
	\end{tikzpicture}
%		\begin{tikzpicture}[scale=0.84]
%		\tikzset{hexagon/.style={draw, regular polygon, regular polygon sides=6, minimum size=0.84cm, inner sep=0pt, rotate=30}}
%		\drawPascalTrap49
%			\foreach \i in {5,6,7,8} {\drawHexNode9{\i}{black!35}}
%		\foreach \i in {4,5,6,7,8}{\drawHexNode{\i}4{black!10}}
%		\foreach \i in {4,5,6,7,8}{\drawhexarrowsmultrev{\i}4011}
%	\foreach \i/\j in {5/5,6/5,6/6, 7/5,7/6,7/7,8/5,8/6,8/7}{\drawhexarrowsmultrev{\i}{\j}111}
%	\hextext444\hextext54{10}\hextext64{20}\hextext74{35}\hextext84{56}
%	\hextext554\hextext65{10}\hextext75{20}\hextext85{56}
%	\drawhexarrowsmultrev88101
%	\end{tikzpicture}
	\caption{Theorem \ref{thm:Hockey variant} for $m=1,n=5$ (left) and Theorem~\ref{thm:hockey gen} for $l=m=3,n=8$ (right)\label{fig:Hockey variant}}
\end{figure}

Now, let us generalize this further. Until now we have considered the weight sequences $1=\binom k0$ and $k=\binom k1$. If we consider more generally weights of the form $\binom kl$ for some fixed $l\in\mathbb N$, we get the following identity:

\begin{theorem}\label{thm:hockey gen}
	For all $n,m,l\in\mathbb N$ such that $l,m\leq n$ we have:
	$$\sum\limits_{k=\max\{l,m\}}^n\binom kl\binom km=\sum\limits_{k=0}^l(-1)^k\binom{n+1}{m+k+1}\binom{n-k}{l-k}$$
\end{theorem}

%If $l\geq m$, then $\binom kl=0$ für $k<l$, so that the summation actually starts at $k=l$.
For small values of $l$, the right-hand side is simpler than the left-hand side. 
The proof is essentially the same as the one of Theorem~\ref{thm:Hockey variant} but we need more columns to balance the weights. Hence, the right-hand side becomes a more complex alternating sum compared to Theorem~\ref{thm:Hockey variant}. The case $l=m=3$ and $n=8$ is depicted in Figure~\ref{fig:Hockey variant} (right).
We leave it as an exercise to our readers to figure out why this works in general.
%\begin{proof}
%The proof is essentially the same as the one of Theorem~\ref{thm:Hockey variant}. We start with the right-hand side and consider the entries of the form $\binom{n+1}{m+k+1}$ in the $(n+1)$-th row of Pascal's triangle, equipped with weights given bei the alternating entries of the $(n-l)$-th column (in reverse order). Using the weight rule, we move upward until we hit the $m$-th column. Since by the weight rule we subtract neighboring column entries, each row's weights are alternating entries of a column (in decreasing order from the left to the right). In the $(m+1)$-th column we thus obtain weights given by the 
%\end{proof}

We can also consider columns with weights given by another column in reverse order. This gives us a particularly nice identity that can be found in \cite[p. 148]{MR554488}, which Grinberg \cite{Grinberg} refers to as \textit{upside-down Chu-Vandermonde identity}. Using variable transformations we can also find it to be a special case of the \textit{Rothe-Hagen identity} \cite{MR75170}. 

\begin{theorem}[Upside-down Chu-Vandermonde identity]\label{thm:Vandermonde upside down}
	Let $l,m,n\in \mathbb N$ be natural numbers such that $l+m\leq n$. Then we have the following identity:
	$$\sum_{k=m}^{n-l}\binom km\binom{n-k}l=\binom{n+1}{l+m+1}$$
\end{theorem}

Displayed is the case $n=7, m=1$ and $l=3$.

\begin{proof}
We start with the right-hand side, i.e., with $\binom{n+1}{l+m+1}$ with weight $1$ (see Figure~\ref{fig:Vandermonde upside down}, marked dark gray). Now, we use the weight rule and work our way upward. As in the proof of Theorem~\ref{thm:Lagrange}, this generates  an inverted (and shifted) version of Pascal's triangle in rows $n+1$ through $n-l+1$.  However, now we continue to move further upward, but we stop on the left, whenever we hit the $m$-th column (marked light gray). By construction, the $(l+m+1-j)$-th column obtains weights given by the $j$-th column in reverse order. Hence, the $(m+1)$-st column obtains the reversed weights of the $l$-th column, and so does the $m$-th column, since we do not apply the weight rule to our final column anymore. This gives us the left-hand side of our identity.

	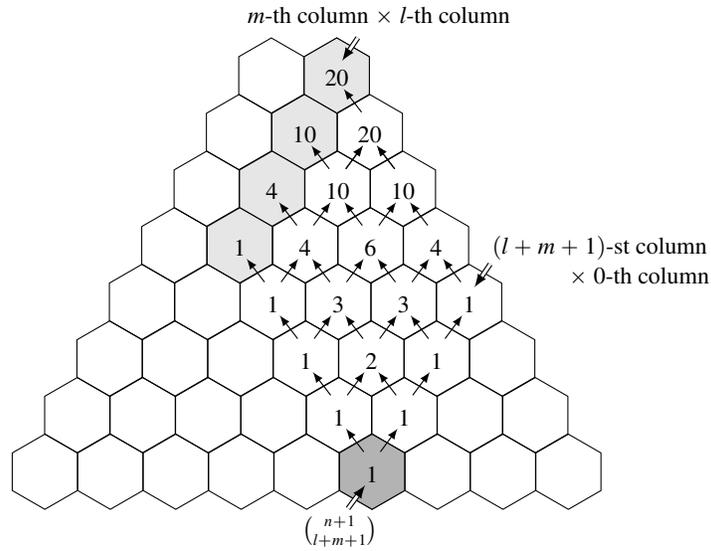
\begin{figure}[h!]\centering
		\begin{tikzpicture}
			\tikzset{hexagon/.style={draw, regular polygon, regular polygon sides=6, minimum size=1cm, inner sep=0pt, rotate=30}}
			
			% Create Pascal's Triangle in hexagon shape
			\drawPascalTrap18
			\foreach \i in {1,2,3,4}{
				\drawHexNode{\i}1{black!10}}
			\drawHexNode85{black!30}
			\foreach \i in {2,3,4}{
				\foreach \j in {2,3,...,\i} {
					\ifnum \i=2
					\drawhexarrowsmultrev{\i}2111
					\else
					\drawhexarrowsmultrev{\i}{\j}111\fi}}
			\foreach\i/\j in {5/3,5/4,6/4}{
				\drawhexarrowsmultrev{\i}{\j}111}
			\foreach \i in {1,2,3,4}{
				\drawhexarrowsmultrev{\i}1011}
			\foreach \i/\j in {5/2,6/3,7/4}{\drawhexarrowsmultrev{\i}{\j}011}
			\foreach \i/\j in {5/5,6/5,7/5}{\drawhexarrowsmultrev{\i}{\j}101}
			\hextext11{20}\hextext21{10}\hextext22{20}\hextext31{4}\hextext32{10}\hextext33{10}\hextext411\hextext424\hextext43{6}\hextext44{4}\hextext521\hextext533\hextext543\hextext55{1}\hextext63{1}\hextext64{2}\hextext65{1}\hextext74{1}\hextext75{1}\hextext85{1}
			\hextext{-0.1}{1.1}{\text{$m$-th column $\times$ $l$-th column}}
			\drawhexarrowsmultrev85102
			\hextext4{6.5}{\text{$(l+m+1)$-st column}}
			\hextext{4.5}{7.35}{\text{$\times$ $0$-th column}}
			\drawhexarrowsmult45102\drawhexarrowsmult01102
			\hextext95{\binom{n+1}{l+m+1}}
		\end{tikzpicture}\caption{Proof of Theorem \ref{thm:Vandermonde upside down} for $n=7,m=1$ and $l=3$\label{fig:Vandermonde upside down}}
	\end{figure}
\end{proof}

While weights of the form $(-1)^k$ give a nice result for weighted row sums, the corresponding weighted column sum does not yield a similarly nice formula. Instead, we consider weights of the form $2^{n-k}$. As we will see, in this case we get a nicer result if we consider two intersecting columns rather than one column. 

\begin{theorem}[Boscarol's identity]
	For all $m,n\in \mathbb N$ such that $m\leq n$ we have:
		$$\sum\limits_{k=m}^n\binom km2^{n-k}+\sum\limits_{k=n-m}^n\binom k{n-m}2^{n-k}=2^{n+1}$$
%	$$\sum\limits_{k=m}^n\frac{\binom km}{2^k}+\sum\limits_{k=n-m}^n \frac{\binom k{n-m}}{2^k}=2$$
\end{theorem} 

This result is in fact just a reformulation of a theorem due to Boscarol which was first proved in 1982 (see \cite{MR668243}). The proof is depicted in Figure~\ref{fig:Boscarol}.
		\begin{figure}[h!]\centering
	\begin{tikzpicture}[scale=1]
		\tikzset{hexagon/.style={draw, regular polygon, regular polygon sides=6, minimum size=1cm, inner sep=0pt, rotate=30}}	\drawPascalTrap37
		\foreach \i in {0,1,2,4,5,6,7} {\drawHexNode7{\i}{black!25}
			\hextext7{\i}{2}}
		\foreach \i/\j in {4/0,5/1,6/2,7/3,6/3,5/3,4/3,3/3}{\drawHexNode{\i}{\j}{black!10}}
		\foreach \i/\j in {5/0,6/0,6/1,4/4,5/4,5/5,6/4,6/5,6/6}{\drawhexarrowsmultrev{\i+0.07}{\j+0.07}111}
		\foreach \i/\j in {4/0,5/1,6/2}{\drawhexarrowsmultrev{\i+0.07}{\j+0.07}101}
		\foreach \i/\j in {3/3,4/3,5/3,6/3}{\drawhexarrowsmultrev{\i}{\j}011}
\hextext604\hextext614\hextext622\hextext632\hextext644\hextext654\hextext664
\hextext508\hextext514\hextext534\hextext548\hextext558\hextext408\hextext43{8}\hextext44{16}\hextext33{16}
		\hextext73{2\cdot1}
		\hextext7{8.5}{\text{$n$-th row}}
		\hextext{2.9}{-1.7}{\binom{n-m}0=\binom40}\drawhexarrowsmult3{-1}012
		\hextext{2.9}{4.5}{\binom mm=\binom33}\drawhexarrowsmulthor34102
		\hextext8{3}{\binom nm=\binom73}\drawhexarrowsmultrev73102
	\end{tikzpicture}
%		\begin{figure}[h!]\centering
%	\begin{tikzpicture}[scale=1.2]
%		\tikzset{hexagon/.style={draw, regular polygon, regular polygon sides=6, minimum size=1.2cm, inner sep=0pt, rotate=30}}	\drawPascalTrap37
%		\foreach \i in {0,1,2,4,5,6,7} {\drawHexNode7{\i}{black!25}
%			\hextext7{\i}{\frac1{2^6}}}
%		\foreach \i/\j in {4/0,5/1,6/2,7/3,6/3,5/3,4/3,3/3}{\drawHexNode{\i}{\j}{black!10}}
%		\foreach \i/\j in {5/0,6/0,6/1,4/4,5/4,5/5,6/4,6/5,6/6}{\drawhexarrowsmultrev{\i+0.07}{\j+0.07}111}
%		\foreach \i/\j in {4/0,5/1,6/2}{\drawhexarrowsmultrev{\i+0.07}{\j+0.07}101}
%		\foreach \i/\j in {3/3,4/3,5/3,6/3}{\drawhexarrowsmultrev{\i}{\j}011}
%		\hextext40{\frac1{2^4}}\hextext50{\frac1{2^4}}\hextext54{\frac1{2^4}}\hextext55{\frac1{2^4}}\hextext44{\frac1{2^3}}\hextext43{\frac1{2^4}}\hextext33{\frac1{2^3}}
%		\hextext60{\frac1{2^5}}\hextext61{\frac1{2^5}}\hextext64{\frac1{2^5}}\hextext65{\frac1{2^5}}\hextext66{\frac1{2^5}}\hextext51{\frac1{2^5}}\hextext53{\frac1{2^5}}
%		\hextext63{\frac1{2^6}}\hextext62{\frac1{2^6}}\hextext73{2\cdot\frac1{2^7}}
%	\end{tikzpicture}
	\caption{Proof of Boscarol's identity\label{fig:Boscarol} for $m=3$ and $n=7$}
\end{figure}
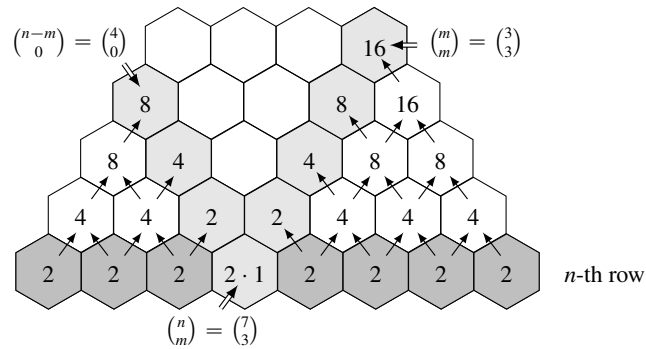

\begin{proof}
By Theorem~\ref{thm:row sum}, the sum of the entries in the \( n \)-th row of Pascal’s triangle is \( 2^n \). Thus, we can interpret the right-hand side of the identity as a weighted row sum, where each entry has weight $2$ (see Figure~\ref{fig:Boscarol}, marked dark gray). The left-hand side corresponds to the sum of two weighted columns: one starting at \( \binom{m}{m} \), and the other at \( \binom{n - m}{n - m} \). Using the symmetry of Pascal’s triangle, we reflect the second column to start at \( \binom{n - m}{0} \) and proceed downward to the right. Both columns then intersect in the \( n \)-th row at \( \binom{n}{m} \) (see Figure~\ref{fig:Boscarol}, marked light gray). Now, we apply the weight rule to move upward from the \( n \)-th row until we reach the diagonals. At each step, this corresponds to adding two weights of the form $2^k$, which results in a weight of $2^k+2^k=2^{k+1}$. Clearly, we obtain the left-hand side.
\end{proof}

%$$\sum\limits_{k=0}^m\frac{\binom{n+k}k}{2^{n+k}}+\sum\limits_{k=0}^n\frac{\binom{n+m-k}m}{2^{n+m-k}}=2$$

Finally, we will choose binomial coefficients of the $n$-th row of Pascal's triangle as weights of the entries of a column. 

\begin{theorem}\label{thm:hor}
	For all $n\in\mathbb N$ we have:
	$$\sum\limits_{k=n}^{2n}(-1)^k\binom kn\binom n{k-n}=1$$
\end{theorem}

To prove Theorem~\ref{thm:hor}, we will use a subtractive version of Pascal’s rule, along with a modified weight rule (see Figure~\ref{fig:sub Pascal}). Specifically, Pascal’s rule becomes  
\[
\binom{n}{k} = \binom{n+1}{k+1} - \binom{n}{k+1}.
\]  
Under the subtractive weight rule, the weight of an entry is obtained by subtracting the weight of the entry to its left from the weight of the entry diagonally above to the left. 
This has the effect that the weighted column sum always remains the same.
 As before, the weight rule applies to entire columns; however, since columns are infinite, we restrict attention to finite segments bounded by zero weights.

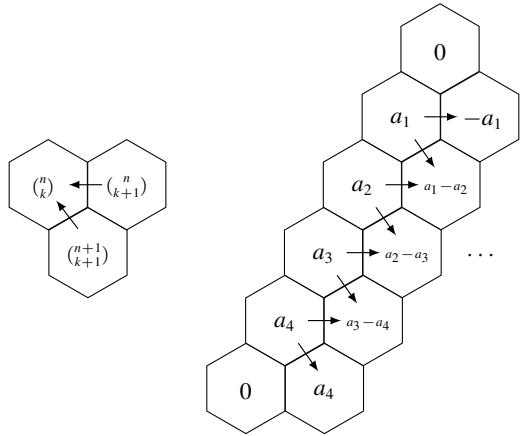
\begin{figure}[h!]\centering
	\begin{tikzpicture}[scale=1.2]
		\tikzset{hexagon/.style={draw, regular polygon, regular polygon sides=6, minimum size=1.2cm, inner sep=0pt, rotate=30}}
		\drawHexNode2{-3}{white}\drawHexNode1{-4}{white}\drawHexNode1{-3}{white}\hextext2{-3}{\text{\scriptsize{$\binom {n+1}{k+1}$}}}\hextext1{-4}{\text{\scriptsize{$\binom n{k}\;\;$}}}\hextext1{-3}{\text{\scriptsize{$\binom n{k+1}$}}}\drawhexarrowsmulthor1{-2.97}101\drawhexarrowsmultrev1{-4}011
		
		\foreach \i/\j in {-1/0,0/0,1/0,2/0,3/0, 4/0,0/1,1/1,2/1,3/1,4/1}{\drawHexNode{\i}{\j}{white}}
		\hextext{-1}00\hextext00{a_1}\hextext10{a_2}\hextext20{a_3}\hextext30{a_4}\hextext400
		\hextext41{a_4}\hextext31{\text{\tiny{\;\;\;$a_3\!-\!a_4$}}}\hextext21{\text{\tiny{\;\;\;$a_2\!-\!a_3$}}}\hextext11{\text{\tiny{\;\;\;$a_1\!-\!a_2$}}}\hextext01{\;-a_1}\hextext22{\cdots}\drawhexarrowsmulthor3{-0.05}011
		\drawhexarrowsmulthor2{-0.05}011\drawhexarrowsmulthor1{-0.05}011\drawhexarrowsmulthor0{-0.05}011
		\foreach \i in {0,1,2,3}{\drawhexarrowsmult{\i}0011}
	\end{tikzpicture}
	\caption{Subtractive version of Pascal's rule (left) and the weight rule (right)\label{fig:sub Pascal}}
\end{figure}

\begin{proof}[Proof of Theorem~\ref{thm:hor}]
We interpret the right-hand side as \( 1 = \binom{n}{0} \), that is, as the leftmost entry in the \( n \)-th row of Pascal’s triangle. Then, using the subtractive weight rule, we move column-wise to the right (see Figure~\ref{fig:hor}, left) so that the column sums remains invariant. By the subtractive weight rule, the weights in the \( l \)-th column correspond to the alternating entries of the \( l \)-th row of Pascal's triangle. Thus, in the \( n \)-th column, we recover the alternating entries of the \( n \)-th row, which matches the left-hand side of the identity.
\end{proof}

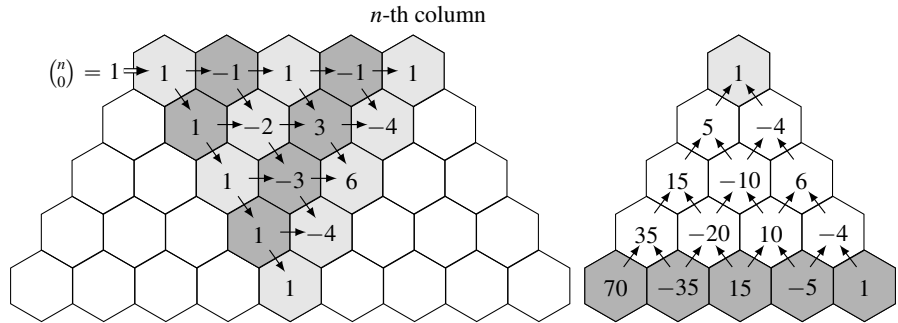
\begin{figure}[h!]\centering
\begin{tikzpicture}[scale=0.95]
	\tikzset{hexagon/.style={draw, regular polygon, regular polygon sides=6, minimum size=0.95cm, inner sep=0pt, rotate=30}}
	\drawPascalTrap48
	\foreach \i/\j in {4/0,4/2,5/2,6/2,4/4,5/4,6/4,7/4,8/4}{\drawHexNode{\i}{\j}{black!10}}
	\foreach \i/\j in {4/1,5/1,4/3,5/3,6/3,7/3}{\drawHexNode{\i}{\j}{black!30}}
	\foreach \i/\j in {4/0,4/1,4/2,4/3,5/1,5/2,5/3,6/2,6/3,7/3}{\drawhexarrowsmulthor{\i}{\j}011\drawhexarrowsmult{\i}{\j}011}
	\hextext401\hextext41{-1}\hextext421\hextext43{-1}\hextext441\hextext511\hextext52{-2}\hextext533\hextext54{-4}\hextext621\hextext63{-3}\hextext646\hextext731\hextext74{-4}\hextext841
	\hextext4{-1.3}{\binom n0=1}
	\drawhexarrowsmulthor4{-1}012
	\hextext{2.9}{3.7}{\text{$n$-th column}}
\end{tikzpicture}\;
\begin{tikzpicture}[scale=0.95]
	\tikzset{hexagon/.style={draw, regular polygon, regular polygon sides=6, minimum size=0.95cm, inner sep=0pt, rotate=30}}
	\drawPascalTriangle4
	\foreach \i in {0,1,2,3,4}{\drawHexNode4{\i}{black!30}}
	\drawHexNode00{black!10}
	\foreach \i/\j in {0/0,1/0,1/1,2/0,2/1,2/2,3/0,3/1,3/2,3/3}{\drawhexarrowsmultrev{\i}{\j}111}
	\hextext40{70}\hextext41{-35}\hextext42{15}\hextext43{-5}\hextext44{1}\hextext30{35}\hextext31{{-20}}\hextext32{10}\hextext33{-4}\hextext20{15}\hextext21{-10}\hextext226\hextext105\hextext11{-4}\hextext001
\end{tikzpicture}\caption{Proof of Theorem~\ref{thm:hor}\label{fig:hor} for $n=4$ as a weighted column sum (left) and weighted row sum (right)}
\end{figure}

	By shifting the rotated alternating version of Pascal's triangle in Figure \ref{fig:hor} (left), this proof can easily be generalized in order to obtain the identity
$$\sum\limits_{k=p}^{p+n}(-1)^{p+n-k}\binom km\binom n{k-p}=\binom{p}{m-n},$$
which was formulated by Knuth \cite{MR3077152} and recently proven combinatorially in \cite{MR4852380}.

Note that Theorem~\ref{thm:hor} can be equivalently stated as  
\[
\sum\limits_{k=0}^n(-1)^k\binom{n}{k}\binom{2n-k}{n} = 1.
\]  
In this form, the left-hand side can be interpreted as a weighted row sum in Pascal’s triangle. This perspective leads to an alternative proof of the identity, which is depicted in Figure~\ref{fig:hor} (right). This proof uses essentially the same idea as the proof of Theorem~\ref{thm:Alternating binomial row sum}.

	\section{Directions for further explorations}
So far, we have presented selected examples that illustrate our technique. However, there are many more applications that we plan to explore in future articles \cite{Krapf1,Krapf2}.

First, there is great flexibility in choosing weight sequences -- for both weighted row sums and weighted column sums. Arrow-chasing is particularly effective when the weights follow a recursive definition. A notable example is the Fibonacci sequence. For instance, using arrow-chasing, we can show \cite{Krapf1}:
\[
\sum\limits_{k=0}^n \binom{n}{k} F_k = F_{2n}.
\]
In fact, as we show in \cite{Krapf1}, arrow-chasing even enables us to discover new identities.
Fibonacci weights also lead to elegant identities for weighted column sums. One striking example, recently posed as a problem in the \textit{Fibonacci Quarterly} \cite{FibQuart}, is:
\[
\sum\limits_{k=0}^n \binom{m}{k} (-1)^{n+k} F_{m - 2k}
= 
\sum\limits_{k=n}^{m-1} \binom{k}{n} F_{k - 2n - 1}.
\]
We leave it as a challenge to the reader to find an arrow-chasing proof.

Second, while we have focused on weighted row and column sums, one can also consider {weighted diagonal sums}. If all weights are 1, we obtain the Fibonacci sequence. Choosing other weight sequences yields many other recursive sequences, whose properties can be shown via arrow-chasing \cite{Krapf2}.

Third, arrow-chasing is not limited to Pascal’s triangle. It also applies to other triangular arrays. A particularly rich example is the {Hosoya triangle}, which allows for proving many Fibonacci identities. But many more structures are waiting to be explored: the Bell triangle, the Stirling triangle, the Leibniz triangle, and others. We hope that our readers will be inspired to explore these fascinating triangles.

	\section{Conclusion}
	In this paper, we introduced a novel technique for constructing proofs via arrow-chasing in Pascal's triangle. This approach offers a variety of new visual proofs for many well-known summation formulas involving binomial coefficients. Since the technique relies solely on Pascal's rule and the corresponding rule for weights, our arguments can be converted into algebraic proofs. However, one significant advantage of this method is its accessibility: it requires only basic arithmetic, making it approachable for a wide audience. Moreover, unlike many algebraic proofs, arrow-chasing not only proves an identity but also provides intuitive insight into \textit{why} the identity holds. While combinatorial proofs (such as those presented in \cite{MR1997773}) also have this explanatory aspect, our method is often simpler, especially when dealing with alternating sums. 
	
%	We showed a few selected examples of arrow-chasing in this paper, but the technique has plenty of other applications. For example, it can also be used in related number arrays such as the Bell triangle or the Catalan triangle. We plan to explore these possibilities in a subsequent article.

	\bibliographystyle{vancouver}
	\bibliography{References.bib}

%\begin{biog}
%\item [Author Name 1] (MR ID or ORCID) Insert author bio here.   %%%%  Comment out in your initial submission
%                                                                                          %%%%  Fill in this information and resubmit once
%\item [Author Name 2] (MR ID or ORCID) Insert author bio here.    %%%%  you receive your provisional accept letter.
%\end{biog}

\end{document}